\documentclass[preprint]{article}
\usepackage{amsthm,amsmath,amssymb}
\usepackage{epsfig}
\usepackage{graphicx}
\usepackage{float}
\usepackage{caption}
\usepackage{subfigure}
\usepackage{enumerate}
\usepackage{epstopdf,color}
\usepackage{multirow}
\usepackage[square,sort,comma,numbers]{natbib}
\usepackage{listings}
\usepackage{bm}
\usepackage{bbm}
\usepackage{authblk}

\usepackage{indentfirst}

\allowdisplaybreaks 

\usepackage[top=1in, bottom=1in, left=1.4in, right=1.4in]{geometry}
\usepackage{hyperref}
\hypersetup{
 colorlinks=true,
 linkcolor=blue,
 filecolor=blue,      
 urlcolor=blue,
 citecolor=blue,
}

\newcommand{\bbE}{{\bf E}}

\newcommand{\bbC}{{\bf C}}

\newcommand{\bbm}{{\bf m}}

\newcommand{\bqn}{\begin{eqnarray*}}
	\newcommand{\eqn}{\end{eqnarray*}}

\newcommand{\bqa}{\begin{eqnarray}}
\newcommand{\eqa}{\end{eqnarray}}

\newtheorem{thm}{Theorem}[section]

\newtheorem{remark}[thm]{Remark}
\newtheorem{lemma}[thm]{Lemma}

\RequirePackage{amsmath,amssymb,graphicx,lscape,txfonts}
\numberwithin{equation}{section}

\numberwithin{equation}{section}

\title{The limiting spectral distribution of large dimensional general information-plus-noise type matrices}

\author{Huanchao Zhou} 
\author{Zhidong Bai}
\author{Jiang Hu}
\affil{School of Mathematics and Statistics, Northeast Normal University, China}
\date{}
\begin{document}
\maketitle

\begin{abstract}
	
	Let $ X_{n} $ be $ n\times N $ random complex matrices, $ R_{n} $ and $ T_{n} $ be  non-random complex matrices with dimensions $n\times N$ and $n\times n$, respectively. We assume that the entries  of $ X_{n} $ are independent and identically distributed, $ T_{n} $ are nonnegative definite Hermitian matrices and  
	$T_{n}R_{n}R_{n}^{*}= R_{n}R_{n}^{*}T_{n} $. 
	The general information-plus-noise type matrices are defined by $ C_{n} =\frac{1}{N} T_{n}^{\frac{1}{2}} \left( R_{n} +X_{n}\right) \left( R_{n}+X_{n}\right)^{*}T_{n}^{\frac{1}{2}} $. 
	In this paper, we establish the limiting spectral distribution of the large dimensional  general information-plus-noise type matrices $C_{n}$. Specifically, we show that as $n$ and $N$ tend to infinity proportionally, the empirical distribution of the eigenvalues of $C_{n}$ converges weakly to a non-random probability distribution, which is characterized in terms of a system of equations of its Stieltjes transform.

\end{abstract}	

\providecommand{\keywords}[1]
{
	\textbf{\text{Keywords: }} Limiting spectral distribution, Random matrix theory, Stieltjes transform
}


\section{Introduction}

Let $ A $ be an $ n \times n $ matrix  with only real eigenvalues  $ \lambda_{1}, \lambda_{2}, \ldots, \lambda_{\textit{n}} $. Denote $ F^{A} $ be the empirical spectral distribution (ESD) function of $ A $, that is $$ F^{A}(x)=\frac{1}{n} \sum_{i=1}^{n}I_{(\lambda_{i} \leq x)}, $$ 
where $ I_{(\cdotp)} $ represents the indicator function. The limiting spectral distribution (LSD)  is defined as the weak limit of the ESD sequence. For any probability distribution function $ F $, the Stieltjes transform of $ F $ is defined as 
\begin{align}
	\begin{split}
		m_{F}(z)=\int \dfrac{1}{\lambda-z} dF(\lambda), \quad 
		z\in \mathbb{C}^{+}, \label{1.2}
	\end{split}
\end{align}
and $ F $ can be obtained by the inversion formula 
\begin{align}
	F(b)-F(a)=\frac{1}{\pi}\lim_{v\rightarrow 0^{+}}\int_{a}^{b}\Im m_{F}(u+iv)du, \label{1.3}
\end{align}
where $ a $, $ b $ are continuity points of $ F $.

It is well known that find the LSD for random matrices has constituted a basic part of large dimensional random matrices theory (LDRMT).  
Since the famous semicircular law and M-P law were established in \cite{wigner1958distribution} and \cite{marchenko1967eigenvalue} respectively,  many researchers have contributed to its subsequent development. One of the most extensively investigated in LDRMT is the so-called sample covariance type matrix taking the form $ B_{n} =\frac{1}{N} T^{\frac{1}{2}}_{n} X_{n} X^{*}_{n}T^{\frac{1}{2}}_{n} $, where $ ^* $ stands for the complex conjugate transpose, $ X_{n} $ is an $ n \times N $  random matrix with independent and identically distributed (i.i.d.) entries $(x_{ij})$ and $ T_{n} $ is an $ n \times n $ non-random nonnegative definite Hermitian matrix. Results on the LSD of the sample covariance type matrix can be found in \cite{yin1986limiting}  and  \cite{silverstein1995strong}.

For the sample covariance type matrix, most of the existing results  is under the centered condition, that is the entries of $ X_{n} $ are zero mean.  
Actually, the large non-centered random matrices also have a significance which are all their own.
One example is the large dimensional information-plus-noise matrix, which is used in the detection problem of array signal processing. More specific, consider the matrix $$ C_{n} =\frac{1}{N}T_{n}^{\frac{1}{2}} \left(R_{n}+X_{n} \right) \left(R_{n} +X_{n} \right)^{*}T_{n}^{\frac{1} {2}},$$ 
where $ R_{n} $ is an $ n\times N $ non-random matrix. We call this non-centered random matrices $ C_{n} $  the general information-plus-noise type matrices which can form a class of large dimensional random matrices.
When $ T_{n}= \sigma^{2} I_{n} $, the matrix  $ C_{n} $ is known as the information-plus-noise type matrix, which is denoted by $$ D_{n}=\frac{1}{N}(R_{n}+\sigma X_{n})(R_{n}+\sigma X_{n})^{*}. $$  
Dozier and Silverstein in \cite{dozier2007empirical} considered the LSD of this type matrices. 
Since the matrix $ R_{n} $ contains the signal information as transmitted and the additive noise matrix 
$ \sigma X_{n} $ is centered, as the numbers of sensors $ n $ and samples $ N $ tend to infinity, they proved that, almost surely, $ F^{D_{n}} $ converges weakly to a nonrandom distribution $ F $, whose Stieltjes transform $ m=m(z) $ satisfies 
\begin{align}
\begin{split}
m=\int\frac{dH(t)}{\frac{t}{1+\sigma^{2}cm}-(1+\sigma^{2}cm)z+\sigma^{2}(1-\textit{c})}, \label{1.1}
\end{split}
\end{align}
and for any $ z\in \mathbb{C}^{+}\equiv\{z\in \mathbb{C}: \Im z>0\} $,  the solution $ m(z) $ is unique in 
$ \mathbb{C}^{+} $. Here $H$ is the LSD of  $ \frac{1}{N}R_{n}R_{n}^{*} $. 

Although the explicit solution of equation \eqref{1.1} is usually not easy  to obtain, the Stieltjes transform 
$ m(z) $ can reveal much of the behavior of the LSD $ F $. Actually, the Stieltjes transform is one of most effective tools to analyze the LSD.
This method was started in \cite{marvcenko1967distribution,pastur1972spectrum,pastur1973spectra} to studying the spectral analysis problems. 
Many papers have contributions to study the LSD of random matrices using the Stieltjes transform, and some important related work can be found in \cite{silverstein1995strong,silverstein1995empirical,silverstein1995analysis}. 
In terms of application, such as wireless communication, this method also plays a great role in processing multiple input multiple output (MIMO) channels in \cite{couillet2011deterministic} and \cite{wen2012deterministic}, etc.
More important, by the inversion formula \eqref{1.3} of the Stieltjes transform, some analytic behavior of the LSD can be obtained in \cite{bai1998no,bai1999exact,dozier2007analysis}.

In this paper, our main work is to find and character the LSD of $F^{C_{n}}$. 
The method we used  is the Stieltjes transform  which is analytic to that used in \cite{dozier2007empirical}. 
Here we show a property of Stieltjes transform that will be needed later. If $ F $ has nonnegative support, then
\begin{align}
\Im zm_{F}(z)=\int \dfrac{\lambda v} {|\lambda-z|^{2}} d F (\lambda) \geq 0, \label{2.3}
\end{align}
for any $ z=u+ iv\in \mathbb{C}^{+} $.
In addition , if $ n \times n $ matrices \textit{A} have real eigenvalues $ \lambda_{1}, \lambda_{2}, \ldots, \lambda_{p} $, the Stieltjes transform of $ F^{A} $ also can be expressed by
\begin{align}
m_{F^{A}}(z)=\frac{1}{n}\sum_{i=1}^{n}\dfrac{1}{\lambda_{i}-z}=\frac{1}{n}\mathrm{tr}(A-zI_{n})^{-1}, \label{1.5}
\end{align} 
which is suitable for our analysis.

The rest sections of this paper are organized as follows. Our main result is present in section \ref{2}. Section \ref{proof} gives a detailed proof of our main results.  Simulations are conducted to evaluate our work in Section \ref{sim}. And some useful lemmas are presented in Appendix.

\section{Main result} \label{2}
In this section, we  show that as the dimentions tend to infinity proportionally, the ESD of $C_{n}$ converges weakly to a non-random LSD almost surely, which is characterized in terms of a system of equations of LSD its Stieltjes transform. To consider the LSD of the general information-plus-noise type matrices, it is necessary to make the following assumptions.

\begin{itemize}
	\item
		{\bf Assumption (a)}: $ \frac{n}{N}=c_{n}\rightarrow c>0 $, as $\min\{n,N\}\to\infty$.
	\item
	{\bf Assumption (b)}: The entries of  $ X_{n} $ are  i.i.d.,  zero mean and unit variance.
%
%

    \item
	{\bf Assumption (c)}: $ R_{n}R_{n}^{*} $ and $ T_{n} $ are multiplication commutative.
	\item
	{\bf Assumption (d)}: As $\min\{n,N\}\rightarrow\infty$, the two-dimensional distribution function 
	$ H_{n}(x,y)=\frac{1}{n} \sum_{i=1}^{n} \textit{I}_{(t_{i}<x, s_{i}<y )} $, converges weakly to a nonrandom distribution $ H(x,y) $ almost surely, 
	where $ t_{i}, s_{i} $ are the paired eigenvalues of 
	$ T_{n} $ and $ \frac{1}{N}R_{n}R_{n}^{*} $,  respectively.

\end{itemize}

Then, we have the following theorem for the LSD of the general information-plus-noise type matrices $C_{n}$. 

\begin{thm} \label{them1}
	Under Assumptions (a)-(d), almost surely, the ESD of the general information-plus-noise type matrices $C_{n}$ almost surely converges  to a nonrandom LSD  $ F $, whose Stieltjes transform  $ m=m_{F}(z) $ satisfies the equation system
\begin{equation}
	\left\{
	\begin{aligned}
	m=\int\frac{dH(s,t)}{\frac{st}{1+cg}-(1+cmt)z+t(1-c)}, \\
	g=\int\frac{tdH(s,t)}{\frac{st}{1+cg}-(1+cmt)z+t(1-c)}. \label{2.1}
	\end{aligned}
	\right.
\end{equation}
Moreover, for each $ z\in \mathbb{C}^{+} $, $ (m, g) $ is the unique solution to 
$ \eqref{2.1} $ in  $ \mathbb{C}^{+} $.
\end{thm}

\begin{remark}
	 In this theorem, $ g $ is the limit of  $ g_{n}=\frac{1}{n}\mathrm{tr}\left(C_{n}-zI\right)^{-1}T_{n} $ actually. 
\end{remark}

\begin{remark}
	The condition that $ R_{n}R_{n}^{*} $ and $ T_{n} $ are commutative in Assumption (c) is necessary in order to get the expression of LSD. 
	In Assumption (d), the two-dimensional distribution function $ H_{n}(x,y) $ is assigned a mass 
	$ \frac{1}{n} $ at each pair of eigenvalues of the two matrices $ R_{n}R_{n}^{*} $ and $ T_{n} $.
\end{remark}

The proof will be presented in Section \ref{proof}. We introduce some notations used in the following.
For any $ n\times N $ matrix $ A $, $ \mathrm{tr}(A) $ and $ A^{*} $ denote its trace and complex conjugate transpose respectively, and $ ||A|| $ deontes its spectral norm, that is the largest singular value of $ A $. For $ q\in\mathbb{C}^{n} $, 
$ ||q|| $ deontes the Eucliden norm. The constant $ K $, appearing henceforth in some of the expressions is nonrandom and may take on different values from one appearance to the next.

\section{Proof of Theorem \ref{them1}} \label{proof}

In this section, we present the proof of Theorem \ref{them1}.

\subsection{Truncation and centralization}

Following the same steps of truncation, centralization, and rescaling in \cite{dozier2007empirical}, we truncated $ X_{n} $ at $ \pm \ln{n} $ and centralized. 
Since $ R_{n}R_{n}^{*} $ and $ T_{n} $ are multiplication commutative,
we write $ T_{n} $ in its spectral decomposition 
$ T_{n}=U_{1} \rm diag(\textit{t}_{1}, \ldots, \textit{t}_{n}) \textit{U}_{1}^{*} $, 
where $ t_{i}\geq 0 $ are the eigenvalues of $ T_{n} $ and $ U_{1} $ is a unitary matrix. 
Following the trunction technique used in Section 4.2 in \cite{bai2010spectral}, 
let $\hat{T}_{n}=U_{1} \rm diag(\hat{\textit{t}}_{1}, \ldots, \hat{\textit{t}}_{n}) \textit{U}_{1}^{*} $,
where $\hat{\textit{t}}_{i}=t_{i} \textit{I}_{(t_{i}<\tau)} $, where $ \tau $ is pre-chosen constant.
Meanwhile, we write $ \frac{1}{\sqrt{N}}R_{n} $ in its singular value decomposition 
$ \frac{1}{\sqrt{N}}R_{n}=U_{1} \rm diag(\sqrt{\textit{s}_{1}}, \ldots, \sqrt{\textit{s}_{n}}) \textit{V}_{1}^{*} $ 
and let $\frac{1}{\sqrt{N}}\hat{R}_{n}=U_{1} \rm diag(\sqrt{\hat{\textit{s}}_{1}}, \ldots, \sqrt{\hat{\textit{s}}_{n}}) \textit{V}_{1}^{*} $ where 
$ \sqrt{\hat{\textit{s}}_{i}}=\sqrt{\textit{s}_{i}}\textit{I}_{(\sqrt{\textit{s}_{i}} \leq \tau)} $.
For brevity, we still use $ T_{n} $ and $ R_{n} $ for the truncated matrix $ T_{n} $ and $ R_{n} $, and we can consider $ T_{n} $ and $ \frac{1}{N}R_{n}R_{n}^{*} $ are spectral norm bounded. 

Therefore, the proof of Theorem \ref{them1} can be done under the following simplified conditions:
\begin{center}
 (1)  $ |x_{ij}|\leq \ln n $,  
 ~~(2)  $ ||\frac{1}{N}R_{n}R_{n}^{*}||\leq K $ , 
 ~~(3)  $ ||T_{n}||\leq K $ .
\end{center}

\subsection{Completing the proof}

We first provide three lemmas, which provide the theoretical foundational of using Stieltjes transform method to show the convergence of the ESD of random matrices. 
\begin{lemma} \label{lemma3.1}
	(Corollary to Theorem 25.10 of \cite{RN02}) If $ \{F_{n} \} $ is a tight subsequence of probability distributions, and if each subsequence that converges weakly at all converges weakly to the probability distribution $ F $, then $ F_{n} $ converges weakly to $ F $.
\end{lemma}
\begin{lemma} \label{lemma3.2}
    (Theorem 2.3.5 of \cite{lixin2007spectral}) Suppose $ \{F_{n} \} $ is a sequence of probability distributions and $ F $ is a probability distribution. Then $ F_{n} $ converges weakly to $ F $ if and only if  $ m_{F_{n}}(z) \to m_{F}(z)  $ for all $ z \in \mathbb{C}^{+} $. 
\end{lemma}
\begin{lemma} \label{lemma3.3}
	(Lemma 1.1 of \cite{bai2010limiting})
	For any random matrices, let $ F^{A_{n}} $ denote the ESD of $A_{n} $ and $ m_{n}(z) $ its Stieltjes transform. Then, if $ F^{A_{n}}$ is tight with probability one and for each $ z\in \mathbb{C}^{+} $, 
	$ m_{n}(z) $ converges almost surely to a nonrandom limit $ m(z) $ as $ n \rightarrow \infty $, then there exists a nonrandom distribution function $ F $ taking $ m(z) $ as its Stieltjes transform such that with probability one, as $ n \rightarrow\infty$, $ F^{A_{n}} $ converges weakly to $ F $.
\end{lemma}

The above three lemmas ensure that the convergence of ESD can be proved by studying their Stieltjes transformation and the limit of their Stieltjes transforms can be found by showing the convergence. 
Thus, to prove Theorem \ref{them1}, we first need to show that $ \{ F^{C_{n}} \} $ is tight with probability one and then prove the almost surely convergence of the Stieltjes transform 
$ \textit{m}_{{F^{\textit{C}_{n}}}}(\textit{z}) $.

From Dozier and Silverstein \cite{dozier2007empirical}, we can get the almost surely tightness of the sequence
$ \{ F^{D_{n}} \} $, where 
$ D_{n}=\frac{1}{N}\left( R_{n}+\sigma X_{n}\right) \left( R_{n}+\sigma X_{n}\right)^{*} $. 
This together with Assumption (d)  and definition of $T_{n} $ guarantees that the sequence
$ \{ F^{T_{n}} \} $ is tight almost surely. And then applying Lemma \ref{lemma2.2} in Appendix, it gives us that $ \{ F^{C_{n}} \} $ is almost surely tight.

Next, we give some preliminary results for the proof of the convergence of $ m_{{F^{C_{n}}}} (z) $.
In the  sequel, we let $ z=u+iv\in \mathbb{C}^{+} $. 
and  $\bbC_{n}=\frac{1}{N}\left(R_{n}+X_{n}\right)^{*}T\left(R_{n}+X_{n} \right) $.
It is easy to find that the eigenvalues of the matrix $ \bbC_{n} $ are the same as those of the matrix 
$ C_{n} $ except $ |n-N| $ zero eigenvalues,  
and the difference is expressed via their ESD   
\begin{align}
F^{\bbC_{n}}=\left(1-\frac{n}{N} \right)I_{[0,\infty]}+\frac{n}{N} F^{ C_{n}}, \label{3.1}
\end{align}
Therefore, we  obtain that
\begin{align}
\bbm_{n}=-\dfrac{1-c_{n}}{z}+c_{n}m_{n},
\label{3.2}
\end{align}
where $ \bbm_{n}(z)=m_{F^{\bbC_{n}}(z)} $ and $ m_{n}(z)=m_{F^{C_{n}}}(z) $.

For any $ j=1,2,\dots,N $,  we define $ y_{j}=\frac{1}{\sqrt{N}} T_{n}^{\frac{1}{2}} \left( r_{j}+x_{j}\right) $, where $ r_{j} $ and $ x_{j} $ are the columns of $ R_{n} $ and $ X_{n} $, respectively. 
So that $ C_{n}=\sum_{j = 1}^N y_j y_j^{*}$ and 
$ \frac{1}{N} R_{n}R_{n}^{*}=\frac{1}{N}\sum_{j=1}^N r_j r_j^{*} $.
Let $ D=C_{n}-zI $, $ B=K_{n}-zI $, where
$$ K_{n}=\left( \dfrac{1}{1+c_{n} g_{n}} \right) \frac{1}{N}T_{n}^{\frac{1}{2}}R_{n}R_{n}^{*} T_{n}^{\frac{1}{2}}-z\bbm_{n}(z)T_{n}, $$ 
and
$ g_{n}=\frac{1}{n}\mathrm{tr}\left(C_{n}-zI \right)^{-1} T_{n} $. 
Denote $C_{j}=C_{n}-y_{j}y_{j}^{*} $ and $ D_{j}=D-y_{j}y_{j}^{*}=C_{j}-zI $.
Multiplying $ D^{-1} $ to the right on both sides of $ D+zI=\sum_{j=1}^N y_j y_j^{*} $  and by Lemma \ref{lemma2.7}, we have that
\begin{equation*}
I+zD^{-1}=\sum_{j = 1}^N \dfrac{1}{1+y_{j}^{*}D_{j}^{-1}y_{j}}y_{j}y_{j}^{*}D_{j}^{-1}.
\end{equation*}
Then taking the trace on both sides of the above equation,  we obtain
\begin{align*}
\begin{split}
c_{n}+zc_{n}m_{n}=\dfrac{1}{N} \sum_{j=1}^N \dfrac{y_{j}^{*}D_{j}^{-1}y_{j}}{1+y_{j}^{*}D_{j}^{-1}y_{j}} =1-\dfrac{1}{N}\sum_{j=1}^N \dfrac{1}{1+y_{j}^{*}D_{j}^{-1}y_{j}},
\end{split}
\end{align*}
which together with  \eqref{3.2} implies that 
\begin{align}
\begin{split}
\bbm_{n}=-\dfrac{1}{N} \sum_{j=1}^N \dfrac{1} {z(1+y_{j}^{*}D_{j}^{-1}y_{j})}.\label{3.3}
\end{split}
\end{align}
Analogous to  (2.3) of \cite{silverstein1995strong}, for each $ j $, we have that
\begin{align*}
\begin{split}
\Im y_{j}^{*}(\frac{1}{z} C_{j}-I) ^{-1} y_{j} 
&=\frac{1}{2i} y_{j}^{*}((\frac{1}{z} C_{j}-I)^{-1}-(\frac{1}{\overline z} C_{j}-I)^{-1})y_{j}\\
&=\frac{v}{|z|^{2}}y_{j}^{*}(\frac{1}{z} C_{j}-I)^{-1}C_{j}(\frac{1}{\overline z} C_{j}-I)^{-1}y_{j} \geq 0.
\end{split}
\end{align*}
Therefore, we can get 
\begin{align}
\begin{split}
\dfrac{1}{\lvert z(1+\frac{1}{N} (r_{j}+x_{j})^{*} T_{n}^{\frac{1}{2}}D_{j}^{-1}T_{n}^{\frac{1}{2}} (r_{j}+x_{j}))\rvert}\leq\dfrac{1}{v}. \label{3.4}
\end{split}
\end{align}

In the next stage of proof, by decomposing the difference of the inverse and expanding the intermediate factor, we can get 
\begin{align} \label{3.5}
& \ \quad B^{-1}-D^{-1}  \\ \nonumber
&=B^{-1}(D-B)D^{-1} =B^{-1} (C_{n}-K_{n})D^{-1}  \\  \nonumber
&=B^{-1}\left(\dfrac{c_{n}g_{n}}{1+c_{n}g_{n}}\frac{1}{N}T_{n}^{\frac{1}{2}}R_{n}R_{n}^{*}
T_{n}^{\frac{1}{2}}
+\frac{1}{N}T_{n}^{\frac{1}{2}}X_{n}R_{n}^{*}T_{n}^{\frac{1}{2}}
+\frac{1}{N}T_{n}^{\frac{1}{2}}R_{n}X_{n}^{*}T_{n}^{\frac{1}{2}}
+\frac{1}{N}T_{n}^{\frac{1}{2}}X_{n}X_{n}^{*}T_{n}^{\frac{1}{2}} +z\bbm_{n}(z)T_{n}\right)D^{-1}  \\ \nonumber
&=\sum_{j=1}^N B^{-1}\left(\dfrac{c_{n}g_{n}}{1+c_{n}g_{n}}
\frac{1}{N}T_{n}^{\frac{1}{2}}r_{j}r_{j}^{*}T_{n}^{\frac{1}{2}}
+\frac{1}{N}T_{n}^{\frac{1}{2}}x_{j}r_{j}^{*}T_{n}^{\frac{1}{2}}
+\frac{1}{N}T_{n}^{\frac{1}{2}}r_{j}x_{j}^{*}T_{n}^{\frac{1}{2}} 
+\frac{1}{N}T_{n}^{\frac{1}{2}}x_{j}x_{j}^{*}T_{n}^{\frac{1}{2}} 
+\frac{1}{N} z \bbm_{n}(z)T_{n} \right) D^{-1}    \\ \nonumber
&=\sum_{j =1}^N\left( \dfrac{c_{n}g_{n}}{1+c_{n}g_{n}}B^{-1}\frac{1}{N}T_{n}^{\frac{1}{2}}r_{j}r_{j}^{*}
T_{n}^{\frac{1}{2}}D^{-1} +B^{-1}\frac{1}{N}T_{n}^{\frac{1}{2}}x_{j}r_{j}^{*}T_{n}^{\frac{1}{2}}D^{-1} 
+B^{-1}\frac{1}{N}T_{n}^{\frac{1}{2}}r_{j}x_{j}^{*}T_{n}^{\frac{1}{2}}D^{-1} 
\right.  \\ \nonumber &\left. \quad \quad \quad 
+B^{-1}\frac{1}{N}T_{n}^{\frac{1}{2}}x_{j}x_{j}^{*}T_{n}^{\frac{1}{2}}D^{-1} 
+B^{-1}\frac{1}{N}z\bbm_{n}(z)T_{n}^{\frac{1}{2}}T_{n}^{\frac{1}{2}}D^{-1} \right), \nonumber
\end{align}
where $ g_{n}=\frac{1}{n}\mathrm{tr}(C_{n}-zI)^{-1}T_{n}.$

Accoding to the fact \eqref{3.3}, we take the trace of both side \eqref{3.5} and divide it by 
$ n $. It can be seen that
\begin{align}
\begin{split}
& \quad \ \frac{1}{n}\mathrm{tr}\left(K_{n}-zI\right)^{-1}-\frac{1}{n}\mathrm{tr}\left(C_{n}-zI\right)^{-1} \\ 
&=\frac{1}{n}\sum_{j =1}^N \mathrm{tr} \left( \dfrac{c g_{n}}{1+cg_{n}}B^{-1} \frac{1}{N} T_{n}^{\frac{1}{2}} r_{j}r_{j}^{*}T_{n}^{\frac{1}{2}}D^{-1} 
+B^{-1}\frac{1}{N}T_{n}^{\frac{1}{2}} x_{j}r_{j}^{*}T_{n}^{\frac{1}{2}}D^{-1}
+B^{-1}\frac{1}{N}T_{n}^{\frac{1}{2}} r_{j}x_{j}^{*}T_{n}^{\frac{1}{2}}D^{-1} 
\right.\\&\left. \quad \quad \quad \quad 
+B^{-1}\frac{1}{N}T_{n}^{\frac{1}{2}} x_{j}x_{j}^{*}T_{n}^{\frac{1}{2}}D^{-1} 
+B^{-1}\frac{1}{N}z\bbm_{n}(z)T_{n}^{\frac{1}{2}}T_{n}^{\frac{1}{2}}D^{-1} \right) \\
&=\frac{1}{n}\sum_{j =1}^N\left(
\dfrac{c_{n}g_{n}}{1+c_{n}g_{n}}\frac{1}{N}r_{j}^{*}T_{n}^{\frac{1}{2}}D^{-1}B^{-1}T_{n}^{\frac{1}{2}}r_{j}
+\frac{1}{N}r_{j}^{*}T_{n}^{\frac{1}{2}}D^{-1}B^{-1}T_{n}^{\frac{1}{2}}x_{j}
+\frac{1}{N}x_{j}^{*}T_{n}^{\frac{1}{2}}D^{-1}B^{-1}T_{n}^{\frac{1}{2}}r_{j}
\right.\\&\left. \quad \quad \quad \quad 
+\frac{1}{N}x_{j}^{*}T_{n}^{\frac{1}{2}}D^{-1}B^{-1}T_{n}^{\frac{1}{2}}x_{j}
-\dfrac{1}{1+\frac{1}{N}(r_{j}+x_{j})^{*}T_{n}^{\frac{1}{2}}D_{j}^{-1} T_{n}^{\frac{1}{2}}(r_{j}
+x_{j})}\frac{1}{N} \mathrm{tr} T_{n}^{\frac{1}{2}} D^{-1} B^{-1}T_{n}^{\frac{1}{2}} \right). \label{3.6}
\end{split}
\end{align}

In the light of the above definition 
$ D= D_{j}+ y_{j} y_{j}^{*}$ and Lemma \ref{lemma2.7}, we can obtain the identity
\begin{align}
D^{-1}=D_{j}^{-1}-\frac{1}{\alpha^{j}}\frac{1}{N}D_{j}^{-1}T_{n}^{\frac{1}{2}}(r_{j}+x_{j})(r_{j}+x_{j})^{*}
T_{n}^{\frac{1}{2}}D_{j}^{-1},  \label{3.7}
\end{align}
where 
$\alpha^{j}=1+\frac{1}{N} (r_{j}+x_{j})^{*} T_{n}^{\frac{1}{2}}D_{j}^{-1}T_{n}^{\frac{1}{2}}(r_{j}+x_{j}). $ 

Applying the identity $ \eqref{3.7} $, we can write $ \eqref{3.6} $ as
\begin{align}
\frac{1}{n} \mathrm{tr} \left(K-zI \right)^{-1} -\frac{1}{n}\mathrm{tr} \left(C_{n}-zI \right)^{-1}
=\frac{1}{n}\sum_{j=1}^N \left( Q_{1nj}+Q_{2nj}+Q_{3nj}+Q_{4nj}-Q_{5nj} \right), \label{3.8}
\end{align}
where
\begin{align*}
\begin{split}
Q_{1nj}
&=\dfrac{c_{n}g_{n}}{1+c_{n}g_{n}}
\left(\frac{1}{N}r_{j}^{*}T_{n}^{\frac{1}{2}}D_{j}^{-1}B^{-1}T_{n}^{\frac{1}{2}}r_{j}
-\dfrac{\frac{1}{N}r_{j}^{*}T_{n}^{\frac{1}{2}}D_{j}^{-1} y_{j} y_{j}^{*} D_{j}^{-1} B^{-1}
T_{n}^{\frac{1}{2}} r_{j}}{1+y_{j}^{*}D_{j}^{-1}y_{j}} \right)\\
&=\frac{1}{\alpha^{j}}\dfrac{c_{n}g_{n}}{1+c_{n}g_{n}}\left(\alpha^{j}\frac{1}{N} r_{j}^{*} T_{n}^{\frac{1}{2}}D_{j}^{-1} B^{-1} T_{n}^{\frac{1}{2}}r_{j}
-\frac{1}{N}r_{j}^{*}T_{n}^{\frac{1}{2}}D_{j}^{-1}\frac{1}{N}T_{n}^{\frac{1}{2}}\left(r_{j}+x_{j}\right)
\left( r_{j}+x_{j} \right)^{*} T_{n}^{\frac{1}{2}} D_{j}^{-1} B^{-1}T_{n}^{\frac{1}{2}}r_{j}\right),
\\
Q_{2nj}
&=\frac{1}{N}x_{j}^{*}T_{n}^{\frac{1}{2}}D_{j}^{-1}B^{-1}T_{n}^{\frac{1}{2}}r_{j}
-\dfrac{\frac{1}{N}x_{j}^{*}T_{n}^{\frac{1}{2}}D_{j}^{-1} y_{j}y_{j}^{*} D_{j}^{-1}B^{-1}T_{n}^{\frac{1}{2}} r_{j}}{1+y_{j}^{*}D_{j}^{-1}y_{j}}\\
&=\frac{1}{\alpha^{j}}\left(\alpha^{j}\frac{1}{N}x_{j}^{*}T_{n}^{\frac{1}{2}}D_{j}^{-1}B^{-1}
T_{n}^{\frac{1}{2}}r_{j}
-\frac{1}{N}x_{j}^{*}T_{n}^{\frac{1}{2}}D_{j}^{-1}\frac{1}{N}T_{n}^{\frac{1}{2}}\left(r_{j}+x_{j}\right)
\left( r_{j}+x_{j}\right)^{*}T_{n}^{\frac{1}{2}}D_{j}^{-1}B^{-1}T_{n}^{\frac{1}{2}}r_{j}\right),
\\
Q_{3nj}
&=\frac{1}{N}r_{j}^{*}T_{n}^{\frac{1}{2}}D_{j}^{-1} B^{-1}T_{n}^{\frac{1}{2}}x_{j}-\dfrac{\frac{1}{N} r_{j}^{*}T_{n}^{\frac{1}{2}} D_{j}^{-1} y_{j}y_{j}^{*}D_{j}^{-1}B^{-1}
T_{n}^{\frac{1}{2}} x_{j}}{1+y_{j}^{*}D_{j}^{-1}y_{j}}\\
&=\frac{1}{\alpha^{j}}\left(\alpha^{j} \frac{1}{N} r_{j}^{*}T_{n}^{\frac{1}{2}}
D_{j}^{-1}B^{-1}T_{n}^{\frac{1}{2}}x_{j}-\frac{1}{N}r_{j}^{*}T_{n}^{\frac{1}{2}}D_{j}^{-1}
\frac{1}{N}T_{n}^{\frac{1}{2}}\left( r_{j}+x_{j}\right)\left( r_{j}+x_{j}\right)^{*} T_{n}^{\frac{1}{2}}D_{j}^{-1}B^{-1} T_{n}^{\frac{1}{2}}x_{j}\right),
\\
Q_{4nj}
&=\frac{1}{N}x_{j}^{*}T_{n}^{\frac{1}{2}}D_{j}^{-1}B^{-1}T_{n}^{\frac{1}{2}}x_{j}
-\dfrac{\frac{1}{N}x_{j}^{*}T_{n}^{\frac{1}{2}}D_{j}^{-1}y_{j}y_{j}^{*}D_{j}^{-1}B^{-1} 
T_{n}^{\frac{1}{2}} x_{j}}{1+y_{j}^{*} D_{j}^{-1}y_{j}}\\
&=\frac{1}{\alpha^{j}}\left(\alpha^{j}\frac{1}{N}x_{j}^{*}T_{n}^{\frac{1}{2}}D_{j}^{-1} B^{-1}T_{n}^{\frac{1}{2}}x_{j}
-\frac{1}{N}x_{j}^{*}T_{n}^{\frac{1}{2}}D_{j}^{-1}\frac{1}{N}T_{n}^{\frac{1}{2}}
\left( r_{j}+x_{j}\right)\left(r_{j}+x_{j}\right)^{*} T_{n}^{\frac{1}{2}}D_{j}^{-1}B^{-1}
T_{n}^{\frac{1}{2}}x_{j}\right),
\end{split}
\end{align*}
and
\begin{align*}
Q_{5nj} &=\frac{1}{\alpha^{j}}\frac{1}{N}\mathrm{tr}T_{n}^{\frac{1}{2}}D^{-1}B^{-1}T_{n}^{\frac{1}{2}}.
\end{align*}

For convenience, we make the following definitions,
$$\rho_{j}=\frac{1}{N} r_{j}^{*}T_{n}^{\frac{1}{2}}D^{-1}_{j}T_{n}^{\frac{1}{2}}r_{j}, \quad \hat{\rho}_{j}=\frac{1}{N} r_{j}^{*}T_{n}^{\frac{1}{2}}D^{-1}_{j}B^{-1}T_{n}^{\frac{1}{2}}r_{j}, $$
$$\beta_{j}=\frac{1}{N}r_{j}^{*}T_{n}^{\frac{1}{2}}D^{-1}_{j}T_{n}^{\frac{1}{2}}x_{j}, \quad \hat{\beta}_{j}=\frac{1}{N}r_{j}^{*}T_{n}^{\frac{1}{2}}D^{-1}_{j}B^{-1}T_{n}^{\frac{1}{2}}x_{j}, $$
$$\omega_{j}=\frac{1}{N}x_{j}^{*}T_{n}^{\frac{1}{2}}D^{-1}_{j}T_{n}^{\frac{1}{2}}x_{j}, \quad \hat{\omega}_{j}=\frac{1}{N}x_{j}^{*}T_{n}^{\frac{1}{2}}D^{-1}_{j}B^{-1}T_{n}^{\frac{1}{2}}x_{j}, $$
$$\gamma_{j}=\frac{1}{N}x_{j}^{*}T_{n}^{\frac{1}{2}}D^{-1}_{j}T_{n}^{\frac{1}{2}}r_{j}, \quad
\hat{\gamma}_{j}=\frac{1}{N}x_{j}^{*}T_{n}^{\frac{1}{2}}D^{-1}_{j}B^{-1}T_{n}^{\frac{1}{2}}r_{j}, $$
so it is easy to get that
$\alpha^{j}=1+\rho_{j}+\beta_{j}+\gamma_{j}+\omega_{j} $, and the terms
$ Q_{inj}(i=1,\cdots,5) $ can be expressed as
\begin{align*}
&Q_{1nj}=\frac{1}{\alpha^{j}}\dfrac{c_{n}g_{n}}{1+c_{n}g_{n}}
\left[(1+\gamma_{j}+\omega_{j})\hat{\rho}_{j}-(\rho_{j}+\beta_{j})\hat{\gamma}_{j} \right], 
\\
&Q_{2nj}=\frac{1}{\alpha^{j}}\left[
(1+\rho_{j}+\beta_{j})\hat{\gamma}_{j}-(\gamma_{j}+\omega_{j})\hat{\rho}_{j} \right], 
\\
&Q_{3nj}=\frac{1}{\alpha^{j}}\left[ (1+\gamma_{j}+\omega_{j})\hat{\beta}_{j}-(\rho_{j}+\beta_{j})\hat{\omega}_{j} \right], 
\\
&Q_{4nj}=\frac{1}{\alpha^{j}}\left[ (1+\rho_{j}+\beta_{j})\hat{\omega_{\textit{j}}}-(\gamma_{j}+\omega_{j})\hat{\beta}_{j} \right], 
\\
&Q_{5nj} =\frac{1}{\alpha^{j}}\frac{1}{N} \mathrm{tr} T_{n}^{\frac{1}{2}} D^{-1}B^{-1} T_{n}^{\frac{1}{2}}.
\end{align*}

Therefore, after simplificaiton, we get 
\begin{align}
\begin{split}
\eqref{3.8}&=\frac{1}{n} \mathrm{tr} \left(K -zI\right)^{-1} 
-\frac{1}{n} \mathrm{tr} \left(C_{n}-zI \right)^{-1}\\ 
&=\frac{1}{n}\sum_{j=1}^N \frac{1}{\alpha^{j}}\left[ \dfrac{1}{1+c_{n}g_{n}}(c_{n}g_{n}-\omega_{j}) \hat{\rho}_{j}-\dfrac{1}{1+c_{n}g_{n}}\gamma_{j}\hat{\rho}_{j}
+\dfrac{1}{1+c_{n}g_{n}}(\rho_{j}\hat{\gamma}_{j}+\beta_{j}\hat{\gamma_{j}})
\right.\\ &\left. \quad \quad \quad \quad 
+\hat{\beta}_{j}
+\hat{\gamma_{j}}+\hat{\omega}_{j}-\frac{1}{N}\mathrm{tr}T_{n}^{\frac{1}{2}}D^{-1}B^{-1}T_{n}^{\frac{1}{2}} \right] \\ 
&= \frac{1}{n} \sum_{j=1}^N \sum_{i=1}^4 \frac{1}{\alpha^{j}} W_{inj}, \label{3.9}
\end{split}
\end{align}
where 
$$  W_{1nj} =\dfrac{1}{1+c_{n}g_{n}}(c_{n}g_{n}-\omega_{j})\hat{\rho}_{j} , \quad 
W_{2nj}=\dfrac{1}{1+c_{n}g_{n}}(\rho_{j} \hat{\gamma}_{j} +\beta_{j} \hat{\gamma}_{j} -\gamma_{j} \hat{\rho}_{j} ), $$
$$ W_{3nj}=\hat{\beta}_{j} + \hat{\gamma}_{j}  , \quad \quad
W_{4nj}=\hat{\omega}_{j} -\frac{1}{N} \mathrm{tr} T_{n}^{\frac{1}{2}} D^{-1} B^{-1} T_{n}^{\frac{1}{2}}. $$

Next, we prove that 
\begin{equation*}
\begin{aligned} 
\frac{1}{n} \mathrm{tr} \left(K_{n}-zI \right)^{-1} -\frac{1}{n} \mathrm{tr}\left(C_{n}-zI \right)^{-1}
\longrightarrow  0 \quad  as \quad  n\rightarrow\infty .
\end{aligned}
\end{equation*}

The next expressions hold for any  $ j=1,2,\cdots,N,$ and any $ n $.
Similar to the definition \eqref{3.2}, $ B$ and $ g_{n} $,  we make the following definitions,
\begin{eqnarray*}
g_{j}&=&\frac{1}{n} \mathrm{tr} \left(C_{j}-zI \right)^{-1} T_{n},\\
m_{j}&=&m_{F^{C_{j}}}(z),\\
\bbm_{j}&=&-\dfrac{1-c_{n}}{z}+c_{n}m_{j},
\end{eqnarray*}
and
\begin{align*}
\begin{split}
B_{j}=\dfrac{1}{1+c_{n}\frac{1}{n}\mathrm{tr}\left(C_{j}-zI \right)^{-1}T_{n}} \frac{1}{N}T_{n}^{\frac{1}{2}}R_{n}R_{n}^{*}T_{n}^{\frac{1}{2}} -z \bbm_{j}(z)T_{n}-zI.
\end{split}
\end{align*}

According to \eqref{3.4} we can get 
\begin{align}
\begin{split}
\dfrac{1}{\lvert\alpha^{j}\rvert}\leq\dfrac{\lvert z \rvert}{v}, \label{3.10}
\end{split}
\end{align}
and using $ \lVert(A-zI)^{-1}\rVert \leq \dfrac{1}{v} $ for any Hermitian matrix $ A $, so we can get 
\begin{align}
\begin{split}
\lVert D_{j}^{-1}\rVert \leq \frac{1}{v}. \label{3.11}
\end{split}
\end{align}

In order to find a suitable bound of $ \dfrac{1}{\lvert 1+c_{n}g_{n}\rvert} $, we observe that $ T_{n} $ is spectral norm bounded and by \eqref{2.3}, we can get 
\begin{align}
\begin{split}
\dfrac{1}{\lvert1+c_{n}\frac{1}{n}\mathrm{tr}\left(C_{n}-zI \right)^{-1} T_{n}\rvert} 
\leq \dfrac{\lvert z \rvert}{v}, \label{3.12}
\end{split}
\end{align}
and similarly
\begin{align}
\begin{split}
\dfrac{1}{\lvert1+c_{n}\frac{1}{n}\mathrm{tr}\left(C_{j}-zI \right)^{-1} T_{n}\rvert} 
\leq \dfrac{\lvert z \rvert}{v}. \label{3.13}
\end{split}
\end{align}

Suppose $ s $ and $ t $ are the eigenvalues of $ \dfrac{1}{N} RR^{*} $ and $ T $ respectively, so we can easily get the corresponding eigenvalues of $ B $, which is 
$$ \lambda_{B}= \frac{st}{1+c_{n}g_{n}}-(1+c_{n} m_{n} t)z +t (1-c_{n}) . $$
Then according to $ \eqref{2.3} $, we can see that  
\begin{align*}
|\lambda_{B}| \geq |\Im \lambda_{B}|=\left| 
\frac{st \Im g_{n}}{|1+c_{n}g_{n}|^{2}} +ct \Im m_{n}z + v  \right| \geq  v.
\end{align*}
Therefore, we get 
\begin{align}
\begin{split}
\lVert B^{-1}\rVert = \dfrac{1}{|\lambda_{B}|} \leq \frac{1}{v}, \label{3.14}
\end{split}
\end{align}
and similarly
\begin{align}
\begin{split}
\lVert B_{j}^{-1}\rVert \leq \frac{1}{v}. \label{3.15}
\end{split}
\end{align}

From Lemma \ref{lemma2.5}, we have
\begin{equation}
\max_{j \leqslant N}\left|m_{n}-m_{j}\right| \leq \frac{1}{n v}, \label{3.16}
\end{equation}
and
\begin{equation}
\max_{j \leqslant N}\left|g_{n}-g_{j}\right| \leq \frac{\|T_{n}\|}{ nv}\leq \frac{K}{nv}. \label{3.17}
\end{equation}

A simple application of Lemma \ref{lemma2.6} gives
\begin{equation}
\bbE\left\|x_{j}\right\|^{12} \leq  K n^{6}(\ln n)^{12}.  \label{3.18}
\end{equation}

Combining \eqref{3.14}, \eqref{3.15}, \eqref{3.12}, \eqref{3.13} and \eqref{3.16}, we get
\begin{equation}
\begin{aligned}
\left\|B_{j}^{-1}-B^{-1}\right\| 
&=\left\|B_{j}^{-1}\left(B-B_{j}\right) B^{-1}\right\| \leq \frac{1}{v^{2}}\left\|B-B_{j}\right\| \\ 
&=\frac{1}{v^{2}}\left\|\frac{ c_{n} \left|g_{j}-g_{n}\right|}{\left(1+ c_{n} g_{n}\right)
\left(1+ c_{n} g_{j}\right)} \frac{1}{N} T_{n}^{\frac{1}{2}}R_{n}R_{n}^{*}T_{n}^{\frac{1}{2}}
+z(\bbm_{j}(z)-\bbm_{n}(z))T_{n}\right\| \\ 
& \leq \frac{1}{v^{2}}\left(\frac{c_{n}\left|g_{j}-g_{n}\right|}{|1+c_{n}g_{n}| |1+c_{n}g_{j}|}
\left\|\frac{1}{N}T_{n}^{\frac{1}{2}} R_{n}R_{n}^{*}T_{n}^{\frac{1}{2}}\right\|
+\left\|z(\bbm_{j}(z)-\bbm_{n}(z))T_{n}\right\|\right) \\
& \leq  \frac{1}{n v^{3}}\left(\frac{ c_{n}|z|^{2}}{v^{2}} K +|z|\ln n\right)\\
&\leq K \frac{\ln n}{n}.  \label{3.19}
\end{aligned}
\end{equation}

For $ W_{1nj} $, using \eqref{3.7}, Lemma \ref{lemma3.1} and Lemma \ref{lemma2.6}, we can get
\begin{equation*}
\begin{aligned} 
\bbE \left|\omega_{j}- c_{n} g_{n}\right|^{6}
&=\frac{1}{N^{6}} \bbE \left|x^{*}_{j}T_{n}^{\frac{1}{2}}D_{j}^{-1}T_{n}^{\frac{1}{2}}x_{j} -\mathrm{tr}T_{n}^{\frac{1}{2}}D^{-1}T_{n}^{\frac{1}{2}}\right|^{6} \\ 
& \leq \frac{K}{N^{6}} \left(\bbE \left|x_{j}^{*} T_{n}^{\frac{1}{2}} D_{j}^{-1} T_{n}^{\frac{1}{2}}x_{j} -\mathrm{tr}T_{n}^{\frac{1}{2}}D_{j}^{-1} T_{n}^{\frac{1}{2}} \right|^{6}
+\bbE \left|\mathrm{tr} \left(D_{j}^{-1}-D^{-1}\right) T_{n}\right|^{6}\right) \\ 
& \leq \frac{K}{N^{3}}(\ln n)^{12}+\frac{K}{N^{6}} \\ 
& \leq K \frac{(\ln n)^{12}}{N^{3}}.
\end{aligned}
\end{equation*}

For $ W_{2nj} $ , using \eqref{3.11}, Lemma \ref{lemma2.5} and the Cauchy-Schward inequality we get
\begin{equation*}
\begin{aligned} 
\bbE \left|\gamma_{j}\right|^{12}
&=\bbE\left|\frac{1}{N}x_{j}^{*}T^{\frac{1}{2}} D_{j}^{-1} T_{n}^{\frac{1}{2}} r_{j} \right|^{12} =\frac{1}{N^{12}}\bbE \left|x_{j}^{*} T_{n}^{\frac{1}{2}}D_{j}^{-1} T_{n}^{\frac{1}{2}} r_{j} r_{j}^{*} T_{n}^{\frac{1}{2}}D_{j}^{-1*}T_{n}^{\frac{1}{2}}x_{j}\right|^{6}\\
&\leq \frac{K}{N^{12}} \bbE \left|x_{j}^{*}T_{n}^{\frac{1}{2}} D_{j}^{-1}T_{n}^{\frac{1}{2}} r_{j} r_{j}^{*}T_{n}^{\frac{1}{2}}D_{j}^{-1*}T_{n}^{\frac{1}{2}}x_{j}-\mathrm{tr} T_{n}^{\frac{1}{2}}D_{j}^{-1} T_{n}^{\frac{1}{2}}r_{j} r_{j}^{*}T_{n}^{\frac{1}{2}} D_{j}^{-1*} T_{n}^{\frac{1}{2}} \right|^{6}\\
&+\frac{K}{N^{12}} \bbE \left|r_{j}^{*}T_{n}^{\frac{1}{2}}D_{j}^{-1*} T_{n}^{\frac{1}{2}}T_{n}^{\frac{1}{2}} D_{j}^{-1}T_{n}^{\frac{1}{2}} r_{j}\right|^{6}\\
&\leq \frac{K}{N^{3}}(\ln n)^{12}+\frac{K}{N^{6}}\\
&\leq \frac{K}{N^{3}}(\ln n)^{12},
\end{aligned}
\end{equation*}
and similarly
\begin{equation*}
\bbE \left|\beta_{j}\right|^{12} \leq \frac{K}{N^{3}}(\ln n)^{12}.
\end{equation*}
Using \eqref{3.11}, \eqref{3.14}, condition (2) and the Cauchy-Schward inequality we get
\begin{equation*}
\left|\hat{\rho}_{j}\right| \leq K ,
\end{equation*}
and
\begin{equation*}
\left|\rho_{j}\right| \leq K.
\end{equation*} 

For $ W_{3nj} $, using \eqref{3.11}, \eqref{3.15}, \eqref{3.18}, \eqref{3.19}, Lemma \ref{lemma2.7} and the Cauchy-Schward inequality we get
\begin{align*} 
\bbE \left|\hat{\gamma}_{j}\right|^{12}&=\bbE \left|\frac{1}{N} x_{j}^{*} T_{n}^{\frac{1}{2}} D_{j}^{-1} B^{-1} T_{n}^{\frac{1}{2}} r_{j}\right|^{12}\\
&\leq \frac{K}{N^{12}} \bbE \left|x_{j}^{*}T_{n}^{\frac{1}{2}} D_{j}^{-1}\left(B^{-1}-B_{j}^{-1}\right)
T_{n}^{\frac{1}{2}} r_{j}\right|^{12}+\frac{K}{N^{12}} \bbE \left| \left|X_{j}^{*} T_{n}^{\frac{1}{2}} D_{j}^{-1} B_{j}^{-1}T_{n}^{\frac{1}{2}} r_{j}\right|^{2}\right| ^{6}\\
&\leq \frac{K}{N^{12}} \bbE
||x_{j}||^{12} ||r_{j}||^{12}\||D_{j}^{-1}||^{12}||T_{n}||^{12}||B^{-1}-B_{j}^{-1}||^{12}\\
&+\frac{K}{N^{12}} \bbE \left|x_{j}^{*}T_{n}^{\frac{1}{2}} D_{j}^{-1} B_{j}^{-1} T_{n}^{\frac{1}{2}}
r_{j} r_{j}^{*} T_{n}^{\frac{1}{2}}B_{j}^{-1*} D_{j}^{-1*}T_{n}^{\frac{1}{2}}x_{j}
-\mathrm{tr}T_{n}^{\frac{1}{2}} D_{j}^{-1} B_{j}^{-1} T_{n}^{\frac{1}{2}}r_{j} r_{j}^{*} T_{n}^{\frac{1}{2}} B_{j}^{-1*} D_{j}^{-1*}T_{n}^{\frac{1}{2}} \right.\\
&\left.+r_{j}^{*} T_{n}^{\frac{1}{2}}B_{j}^{-1*} D_{j}^{-1*}T_{n}^{\frac{1}{2}}T_{n}^{\frac{1}{2}} D_{j}^{-1} B_{j}^{-1}T_{n}^{\frac{1}{2}} r_{j}\right|^{6}\\
&\leq \frac{K}{N^{12}}(\ln n)^{12}+\frac{K}{N^{12}} \bbE \left|x_{j}^{*}T_{n}^{\frac{1}{2}} D_{j}^{-1} B_{j}^{-1} r_{j} r_{j}^{*}B_{j}^{-1*} D_{j}^{-1*}T_{n}^{\frac{1}{2}}x_{j}
-\mathrm{tr}T_{n}^{\frac{1}{2}}D_{j}^{-1}B_{j}^{-1}r_{j}r_{j}^{*}
B_{j}^{-1*}D_{j}^{-1*}T_{n}^{\frac{1}{2}} \right|^{6}\\
&+\frac{K}{N^{12}} \bbE \left|r_{j}^{*} B_{j}^{-1*} D_{j}^{-1*}T_{n}^{\frac{1}{2}}T_{n}^{\frac{1}{2}} D_{j}^{-1} B_{j}^{-1} r_{j}\right|^{6}\\
&\leq \frac{K}{N^{3}}(\ln n)^{12}.
\end{align*}
and similarly
\begin{equation*}
\bbE \left|\hat{\beta}_{j}\right|^{12}\leq \frac{K}{N^{3}}(\ln n )^{12}.
\end{equation*}

For $ W_{4nj} $, using \eqref{3.11}, \eqref{3.14}, \eqref{3.15}, \eqref{3.18}, \eqref{3.19} and Lemma \ref{lemma2.5} and Lemma \ref{lemma2.6}, we can get 
\begin{equation*}
\begin{aligned}
&\bbE \left|\hat{\omega}_{j}-\frac{1}{N}\mathrm{tr}T_{n}^{\frac{1}{2}}D^{-1}B^{-1}T_{n}^{\frac{1}{2}} \right|^{6}\\
&=\frac{1}{N^{6}} \bbE\left|x_{j}^{*}T_{n}^{\frac{1}{2}} D_{j}^{-1} B^{-1}T_{n}^{\frac{1}{2}} x_{j}-\mathrm{tr}T_{n}^{\frac{1}{2}} D^{-1} B^{-1}T_{n}^{\frac{1}{2}}\right|^{6}\\
&\leq \frac{K}{N^{6}} \bbE\left|x_{j}^{*}T_{n}^{\frac{1}{2}} D_{j}^{-1}\left(B^{-1}-B_{j}^{-1}\right)
T_{n}^{\frac{1}{2}} x_{j}\right|^{6}
+\frac{K}{N^{6}}\bbE\left|x_{j}^{*} T_{n}^{\frac{1}{2}}D_{j}^{-1} B_{j}^{-1} x_{j}-\mathrm{tr} T_{n}^{\frac{1}{2}}D_{j}^{-1} B_{j}^{-1}T_{n}^{\frac{1}{2}}\right|^{6} \\
&+\frac{K}{N^{6}} \bbE\left|\mathrm{tr} T_{n}^{\frac{1}{2}} D_{j}^{-1}\left(B_{j}^{-1}-B^{-1}\right)
T_{n}^{\frac{1}{2}}\right|^{6}+\frac{K}{N^{6}} \bbE\left|\mathrm{tr}T_{n}^{\frac{1}{2}}\left(D_{j}^{-1}-D^{-1}\right) B^{-1}T_{n}^{\frac{1}{2}}\right|^{6}\\
&\leq K \dfrac{(\ln n)^{12}}{N^{3}}.
\end{aligned}
\end{equation*} 

Thus from the Cauchy-Schwarz inequality and the above bounds we get
$$ \bbE \left|\beta_{j} \hat{\gamma}_{j}\right|^{6} \leq K \frac{(\ln n)^{12}}{N^{3}}, \quad 
\bbE \left|\rho_{j} \hat{\gamma}_{j}\right|^{6} \leq  K \frac{(\ln n)^{6}}{N^{\frac{3}{2}}}, \quad 
\bbE \left|\hat{\rho_{j}} \gamma_{j}\right|^{6} \leq  K \frac{(\ln n)^{6}}{N^{\frac{3}{2}}}. $$

Therefore, we have almost surely, as $n \rightarrow \infty $,
\begin{equation}
\begin{aligned} 
\max_{j\leq N}\max\left\lbrace\left|\frac{\left(c_{n}g_{n}-\omega_{j}\right)\hat{\rho}_{j}}{1+c_{n} g_{n}} \right|, \left|\frac{\gamma_{j} \hat{\rho}_{j}}{1+c_{n}g_{n}}\right|, |\hat{\beta}_{j}|, | \hat{\gamma}_{j}|, \left|\frac{\rho_{j} \hat{\gamma}_{j}}{1+c_{n} g_{n}}\right|, 
\left|\frac{\beta_{j} \hat{\gamma}_{j}}{1+c_{n} g_{n}}\right|,\left|\hat{\omega}_{j}-\frac{1}{N} 
\mathrm{tr} T_{n}^{\frac{1}{2}} D^{-1} B^{-1} T_{n}^{\frac{1}{2}} \right|  \right\rbrace   \stackrel{a.s.}{\longrightarrow} 0. \label{3.20}
\end{aligned}
\end{equation}

Now according to $ \eqref{3.10} $ and $ \eqref{3.20} $, as $ n\rightarrow\infty $, we can get 
\begin{equation}
\begin{aligned} 
\frac{1}{n} \mathrm{tr} \left(K_{n}-zI\right)^{-1} -\frac{1}{n}\mathrm{tr}\left(C_{n}-zI\right)^{-1} \rightarrow 0.  \label{3.21}
\end{aligned}
\end{equation}

%
%
It is argued that $ \{F^{C_{n}}\} $ is tight, the quantity 
$$ \varsigma=\inf_{n\to \infty} \Im m_{F^{C_{n}}}(z) \geq \inf_{n\to \infty} \int 
\dfrac{v d F^{C_{n}} (\lambda)}{2(\lambda^{2}+u^{2})+v^{2}} $$
is almost surely positive.

For each fixed $z$, $ \{m_{n}(z)\} $ is a bounded sequence (bounded in absolute value by 
$ \dfrac{1}{v} $). 
And because of $ \eqref{3.21}$, we know that $ \frac{1}{n} \mathrm{tr} \left(K_{n}-zI\right)^{-1} $ is bounded. 

For any subsequence $ n_{i} $, there is a subsequence $ \{ n^{'}_{i} \} $ of 
$ \{ n_{i} \} $ that $ m_{n^{'}_{i}}(z) $ converges to a limit $ m $. And $ m $ should satisfiy $ \eqref{2.1} $. 
Then as $ n_{i}\rightarrow\infty $, one has 
\begin{align*}
&\frac{1}{n_{i}} \mathrm{tr} (K_{n_{i}}-zI)^{-1}
=\int\dfrac{d H_{n}(s,t)}{\frac{st}{1+c_{n_{i}}g_{n_{i}}}-(1+c_{n_{i}}m_{n_{i}}t)z
+ t(1-c_{n_{i}})} \longrightarrow \int \frac{dH(s,t)}{\frac{st}{1+cg}-(1+cmt)z+t(1-c)} .
\end{align*}

We conclude that for each $ z $ with $ \Im z > 0 $, as $ n \rightarrow \infty $,
$ m_{n}(z)\stackrel{a.s.}{\longrightarrow} m(z) $, and $ m(z) $  satisfies the equation system $ \eqref{2.1} $.

\subsection{Unique solution to \eqref{2.1} } \label{us}

We now prove that a certain type of solution to \eqref{2.1}  is unique.
\begin{thm} \label{thm2}
	Let $ z = u + iv \in \mathbb{C}^{+} $, $m_{1}, m_{2} \in \mathbb{C}^{+} $ and  
	$ g_{1},g_{2} \in \mathbb{C}^{+} $. If both $ m_{1} $, $ g_{1} $ and $ m_{2} $, $ g_{2}$ satisfy 
	$ \eqref{2.1} $ respectively, then 
	$ (m_{1},g_{1} )= (m_{2},g_{2} ) $.
\end{thm}
\begin{proof}

Suppose $m_n\to m $ and $g_n\to g$ along some subsequence, and the set of equations \eqref{2.1}  have two different solutions $( m_1, g_1) \neq (m_2, g_2)$, then we have 
\begin{eqnarray}\label{b3}
m_1-m_2&=&\int\frac{\frac{cts(g_1-g_2)}{(1+cg_1)(1+cg_2)}+ctz({m_1- m_2}) }{\left(
\frac{st}{1+cg_1}-(1+cm_1 t)z +t(1-c)\right) \left(\frac{st}{1+cg_2} -(1+cm_2 t)z +t(1-c) \right)}dH(t,s),\\ \label{4.1}
g_1-g_2&=&\int\frac{\frac{ct^2s(g_1-g_2)}{(1+cg_1)(1+cg_2)}+ct^2z({m_1- m_2})}{{\left(\frac{st}{1+cg_1} -(1+cm_1 t)z +t(1-c)\right)\left(\frac{st}{1+cg_2}-(1+cm_2 t)z+t(1-c) \right)}}dH(t,s).\label{4.2}
\end{eqnarray}

For the convenience of calculation, we have made some definitions as follows.  
Write for $ \ell=1,2 $,
\begin{eqnarray*}
    A_{\ell}(g_{1})&=&\int\frac{\frac{ct^\ell s}{|1+cg_1|^2}} {\left|\frac{st}{1+cg_1}-(1+cm_1 t)z+t(1-c) \right|^2}dH(t,s), \\
	B_\ell(g_{1})&=&\int\frac{ct^\ell |z|}{\left|\frac{st}{1+cg_1} -(1+cm_1 t)z+t(1-c) \right|^2}dH(t,s),\\
	C_\ell(g_{1})&=&\int\frac{t^{\ell-1}} {\left|\frac{st}{1+cg_1}-(1+cm_1 t)z+t(1-c)\right|^2}dH(t,s).
\end{eqnarray*}

By comparing the imaginary parts of \eqref{2.1} and denoting $g=g_r+ig_i$, we obtain 
\begin{eqnarray}\label{b6}
m_i&>& A_1(g)g_i+B_1(g)\frac{(z{m})_i}{|z|},\\ \label{b0}
g_i&>& A_2(g)g_i+B_2(g)\frac{(z{m})_i}{|z|}, 
\end{eqnarray}
which implies
\begin{eqnarray}\label{b7}
g_i&>&\frac{B_2(g)}{1-A_2(g)}\frac{(z{m})_i}{|z|}.
\end{eqnarray}
Substituting (\ref{b7}) into (\ref{b6}), we obtain
\begin{eqnarray}\label{b8}
{m}_i&>& \left(\frac{A_1(g)B_2(g)}{1-A_2(g)}+B_1(g)\right)\frac{(z{m})_i}{|z|}.
\end{eqnarray}
Then by multiplying $\textit{z}$ on both sides of (\ref{2.1}) and comparing their imaginary parts, as before, we obtain 
\begin{equation}\label{b12}
(z{m})_i=\left(\frac{A_1(g)B_2(g)}{1-A_2(g)}+B_1(g)\right)|z|m_i.
\end{equation}
Substituting (\ref{b12}) into (\ref{b8}), we obtain 
\[ {m}_i>\left(\frac{A_1(g)B_2(g)}{1-A_2(g)}+B_1(g)\right)^2 {m}_i.\]
Hence, it is clear that 
\begin{equation}
\left(\frac{A_1(g)B_2(g)}{1-A_2(g)}+B_1(g)\right)^2<1.
\end{equation}
The inequality (\ref{b0}) implies that $1-A_2(g)>0$, which, together with the above inequality, implies that 
\begin{equation} \label{b11}
\frac{A_1(g)B_2(g)}{(1-A_2(g))(1-B_1(g))}<1. 
\end{equation}

Next, to complete the proof of uniqueness of solutions, we choose disproved method and we shall derive an inequality which is contradictory to (\ref{b11}).

From (\ref{b3}) and (\ref{4.2}) and applying the Cauchy-Schwarz inequaity, we obtain that 
\begin{eqnarray*}
	|{m}_1-{m}_2|&\le & A_1^{1/2}(g_1)A_1^{1/2}(g_2)|g_1-g_2|+ B_1^{1/2}(g_1)B_1^{1/2}(g_2)|{m}_1-{m}_2|\\
	&\le &\frac{A_1^{1/2}(g_1)A_1^{1/2}(g_2)}{1- B_1^{1/2}(g_1)B_1^{1/2}(g_2)}|g_1-g_2|,\\
	|g_1-g_2|&\le & A_2^{1/2}(g_1)A_2^{1/2}(g_2)|g_1-g_2|+ B_2^{1/2}(g_1)B_2^{1/2}(g_2)|{\textit{m}_1- \textit{m}_2}|\\
	&\le&\frac{B_2^{1/2}(g_1)B_2^{1/2}(g_2)}{1-A_2^{1/2}(g_1)A_2^{1/2}(g_2)}|{m}_1-{m}_2|.
\end{eqnarray*}
By the elementary inequality that $1-ab\ge \sqrt{(1-a^2)(1-b^2)}$, the above two inequalities yield,
\begin{eqnarray*}
	|{m}_1-{m}_2|&\le &\frac{A_1^{1/2}(g_1)A_1^{1/2}(g_2)}{[(1- B_1(g_1))(1-B_1(g_2))]^{1/2}}|g_1-g_2|\\
	|g_1-g_2|&\le & \frac{B_2^{1/2}(g_1)B_2^{1/2}(g_2)}{[(1-A_2(g_1))(1-A_2(g_2))]^{1/2}}|{m}_1-{m}_2|.
\end{eqnarray*}
From the inequalities above, we can easily obtain that 
\begin{equation}\label{b5}
1\le \frac{A_1^{1/2}(g_1)A_1^{1/2}(g_2)B_2^{1/2}(g_1)B_2^{1/2}(g_2)}{[(1-A_2(g_1))(1-A_2(g_2))(1- B_1(g_1)) (1-B_1(g_2))]^{1/2}}.
\end{equation}

Making $g=g_1, g_2$, we reach to a contradiction to the inequality (\ref{b5}). The uniqueness of solutions of equation system \eqref{2.1} is proved.

\end{proof}

\section{Simulation} \label{sim}

In order to verify the correctness of Theorem \ref{them1}, we provide the following simulations to support our theoretical analysis and calculations in this section.  We will show the comparison between the empirical spectral density and the limiting spectral density of the $ F^{C_{n}} $ 
under the following cases.

\begin{figure}[htbp]
	\centering
	\begin{minipage}[t]{0.48\textwidth}
		\centering
		\includegraphics[width=6.2cm]{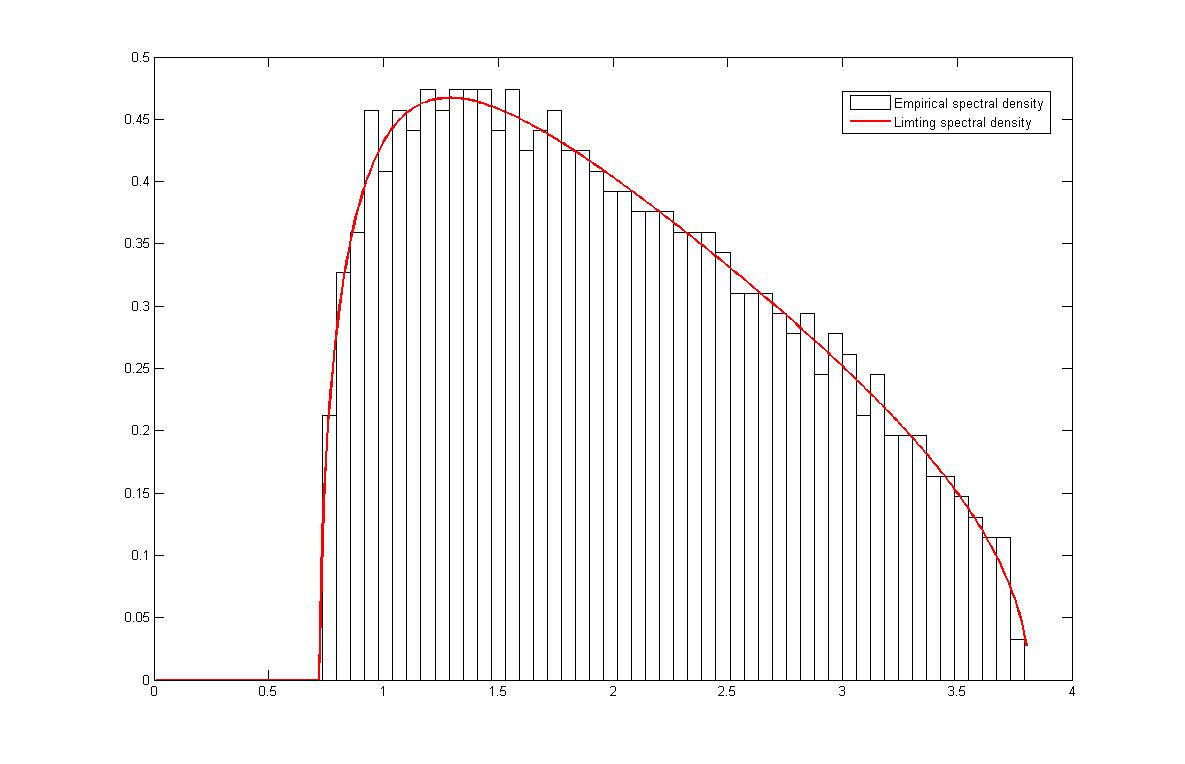} 
		\caption{Case 1: $ \textit{H} = \delta_{(1,1)} $}\label{f1}
	\end{minipage}
	\hspace{.1in}
	\begin{minipage}[t]{0.48\textwidth}
		\centering
		\includegraphics[width=6.2cm]{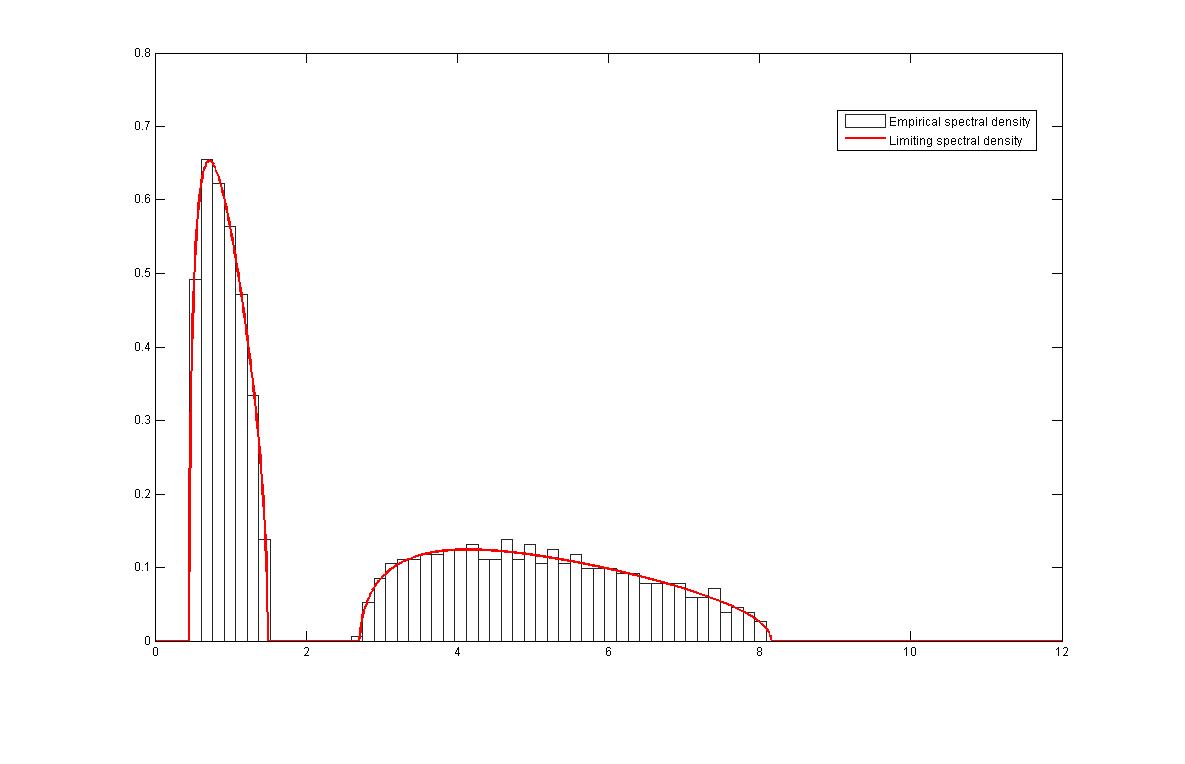}
		\caption{Case 2: $ \textit{H} = \frac{1}{2}\delta_{(0.5,1)}+\frac{1}{2}\delta_{(2.8,1)} $}\label{f2}
	\end{minipage}
\end{figure}

 
In Figure $ \ref{f1} $ and Figure $ \ref{f2} $, they show that the spectrum of the matrices have one point measure and two point mass respectively where $ \delta $ is dirac function. 
We solve $ \textit{m}(\textit{z}) $ from the equation \eqref{2.1} to obtain the density $ f(x)$ by 
$ \eqref{1.3} $.
The histogram represents the empirical spectral density of $ F^{C_{n}} $, and the red line is the density of the LSD calculated by the Stieltjes transform. According to their degree of fit, it implies the correctness of our LSD results.

\appendix
\section{Appendix} \label{Appendix}

There are some lemmas are well known.

\begin{lemma}( Lemma A.1. of Bai and Zhang \cite{bai2010spectral}) \label{lemma2.2}
	Let $ x_{1}, x_{2}, x_{3} $ be arbitrary non-negative numbers. For $ A, B, C $ are square matrices of the same size, then
	$$ F^{\sqrt{(ABC)(ABC)^{*}}}((x_{1}x_{2}x_{3},\infty)) \leq F^{\sqrt{AA^{*}}}((x_{1},\infty))
	+ F^{\sqrt{BB^{*}}}((x_{2},\infty))	+ F^{\sqrt{CC^{*}}}((x_{3},\infty)), $$
	where $ \sqrt{AA^{*}}  $ denotes the matrix drived from $ AA^{*} $ by replacing the eigenvalues in its spectral decomposition the  eigenvalues with their square roots. And  $\sqrt{BB^{*}},\sqrt{CC^{*}}$ are  simliar.
\end{lemma}

\begin{lemma} \label{lemma2.3}
	For $ r\times s $ matrices $ A $ and $ B $, with singular values 
	$ \sigma_{1}\geq\sigma_{2}\geq...\geq\sigma_{q}, \tau_{1}\geq\tau_{2}\geq...\geq\tau_{q} $, where 
	$ q=min\{r,s\} $, we have
	\begin{align*}
	\left( \sum_{i=1}^q(\sigma_{i}-\tau_{i})^{2}\right)^{\dfrac{1}{2}}\leq\|A-B\|_{2},
	\end{align*}
	where $ \|\cdot\|_{2} $ is the  Forbenius norm.
\end{lemma}
\begin{lemma} \label{lemma2.4}
	For $ n\times N $ matrices $ P $ and $ Q $, 
	\begin{align*}
	\|F^{PP^{*}}-F^{QQ^{*}}\|\leq\frac{2}{n} \mathrm{rank}(P-Q).
	\end{align*}
\end{lemma}
\begin{lemma} \label{lemma2.5}
	( Lemma 2.6 of Silverstein and Bai \cite{silverstein1995empirical})  Let $z=u+iv \in \mathbb{C}$ with 
	$A $ $ n\times n $ matrix and $ B $ Hermitian matrix, and $ r \in \mathbb{C}^{n}$. Then
	\begin{equation*}
	\begin{aligned}
	|&\mathrm{tr}\left((B-zI)^{-1}-\left(B+rr^{*}-zI\right)^{-1}\right) A| 
	=\left|\frac{r^{*}(B-z I)^{-1} A(B-z I)^{-1} r}{1+r^{*}(B-z I)^{-1} r}\right| \leq\frac{\|A\|}{\textit{v}}.
	\end{aligned}
	\end{equation*}
\end{lemma}
\begin{lemma} \label{lemma2.6}
	(Lemma 3.1 of Silverstein and Bai \cite{silverstein1995empirical}) Let $C=(c_{ij}), c_{ij}\in\mathbb{C}$, be an $ n\times n $ matrix with $ \|C\| \leq 1 $, 
	and $ Y=\left(X_{1}, \ldots, X_{n}\right)^{T}, X_{i} \in \mathbb{C} $, 
	where the $ X_{i} $'s are i.i.d. satisfying Assumption (a) and $ |X_{11}|\leq \ln n $. Then
	\begin{equation*}
	E\left|Y^{*}CY-\mathrm{tr} C\right|^{6} \leqslant K n^{3}(\ln (n))^{12},
	\end{equation*}
	where the constant $ K $ does not depend on $ n $, $ C $.
\end{lemma}
\begin{lemma} \label{lemma2.7}(Sherman-Morrison formula)
	For $ n\times n$ matrices $ A $ and $ n\times 1 $ vectors $ q $ and $ v $, where $ A $ and $ A+vv^{*} $ are invertible, one has
	\begin{align*}
	q^{*}(A+vv^{*})^{-1}=q^{*}A^{-1}-\dfrac{q^{*}A^{-1}v}{1+v^{*}A^{-1}vv^{*}}^{-1}.
	\end{align*}
	When $q=v $, then
	\begin{align*}
	v^{*}(A+vv^{*})^{-1}=\dfrac{1}{1+v^{*}A^{-1}v}v^{*}A^{-1}.
	\end{align*}
\end{lemma}

\bibliographystyle{plain}	
\bibliography{lsdt621}

\begin{thebibliography}{10}

\bibitem{bai1998no}
Z.D. Bai and J.W. Silverstein.
\newblock No eigenvalues outside the support of the limiting spectral
  distribution of large-dimensional sample covariance matrices.
\newblock {\em The Annals of Probability}, 26(1):316--345, 1998.

\bibitem{bai1999exact}
Z.D. Bai and J.W. Silverstein.
\newblock Exact separation of eigenvalues of large dimensional sample
  covariance matrices.
\newblock {\em The Annals of Probability}, 27(3):1536--1555, 1999.

\bibitem{bai2010spectral}
Z.D. Bai and J.W. Silverstein.
\newblock {\em Spectral analysis of large dimensional random matrices}.
\newblock Springer, 2010.

\bibitem{bai2010limiting}
Z.D. Bai and L.X. Zhang.
\newblock The limiting spectral distribution of the product of the wigner
  matrix and a nonnegative definite matrix.
\newblock {\em Journal of Multivariate Analysis}, 101(9):1927--1949, 2010.

\bibitem{RN02}
P.P. Billingsley.
\newblock {\em Probability and Measure}.
\newblock John Wiley and Sons, 1995.

\bibitem{couillet2011deterministic}
Romain Couillet, M{\'e}rouane Debbah, and J.W. Silverstein.
\newblock A deterministic equivalent for the analysis of correlated mimo
  multiple access channels.
\newblock {\em IEEE Transactions on Information Theory}, 57(6):3493--3514,
  2011.

\bibitem{dozier2007analysis}
R.B. Dozier and J.W. Silverstein.
\newblock Analysis of the limiting spectral distribution of large dimensional
  information-plus-noise type matrices.
\newblock {\em Journal of Multivariate Analysis}, 98(6):1099--1122, 2007.

\bibitem{dozier2007empirical}
R.B. Dozier and J.W. Silverstein.
\newblock On the empirical distribution of eigenvalues of large dimensional
  information-plus-noise-type matrices.
\newblock {\em Journal of Multivariate Analysis}, 98(4):678--694, 2007.

\bibitem{marvcenko1967distribution}
V.A. Mar{\v{c}}enko and L.A. Pastur.
\newblock Distribution of eigenvalues for some sets of random matrices.
\newblock {\em Mathematics of the USSR-Sbornik}, 1(4):457, 1967.

\bibitem{marchenko1967eigenvalue}
V.A. Mar{\v{c}}enko and L.A. Pastur.
\newblock The eigenvalue distribution in some ensembles of random matrices.
\newblock {\em Math. USSR Sbornik}, 1:457--483, 1967.

\bibitem{pastur1972spectrum}
L.A. Pastur.
\newblock On the spectrum of random matrices.
\newblock {\em Teoreticheskaya i Matematicheskaya Fizika}, 10(1):102--112,
  1972.

\bibitem{pastur1973spectra}
L.A. Pastur.
\newblock Spectra of random self adjoint operators.
\newblock {\em Russian mathematical surveys}, 28(1):1--67, 1973.

\bibitem{silverstein1995strong}
J.W. Silverstein.
\newblock Strong convergence of the empirical distribution of eigenvalues of
  large dimensional random matrices.
\newblock {\em Journal of Multivariate Analysis}, 55(2):331--339, 1995.

\bibitem{silverstein1995empirical}
J.W. Silverstein and Z.D. Bai.
\newblock On the empirical distribution of eigenvalues of a class of large
  dimensional random matrices.
\newblock {\em Journal of Multivariate analysis}, 54(2):175--192, 1995.

\bibitem{silverstein1995analysis}
J.W. Silverstein and Choi.
\newblock Analysis of the limiting spectral distribution of large dimensional
  random matrices.
\newblock {\em Journal of Multivariate Analysis}, 54(2):295--309, 1995.

\bibitem{wen2012deterministic}
C.K. Wen, G.M. Pan, K.K. Wong, M.H. Guo, and J.C. Chen.
\newblock A deterministic equivalent for the analysis of non-gaussian
  correlated mimo multiple access channels.
\newblock {\em IEEE Transactions on Information Theory}, 59(1):329--352, 2012.

\bibitem{wigner1958distribution}
E.P. Wigner.
\newblock On the distribution of the roots of certain symmetric matrices.
\newblock {\em Annals of Mathematics}, pages 325--327, 1958.

\bibitem{yin1986limiting}
Y.Q. Yin.
\newblock Limiting spectral distribution for a class of random matrices.
\newblock {\em Journal of multivariate analysis}, 20(1):50--68, 1986.

\bibitem{lixin2007spectral}
L.X. Zhang.
\newblock {\em Spectral analysis of large dimentional random matrices}.
\newblock 2007.

\end{thebibliography}

\end{document}